\documentclass[a4paper,11pt,reqno]{amsart}
\usepackage{amssymb, amsmath, amsfonts, amscd, mathtools}
\usepackage{mathrsfs, latexsym}
\usepackage[all,ps,cmtip]{xy}
\usepackage{array, longtable}
\usepackage{bm, enumitem}
\usepackage{xcolor}
\usepackage{comment}
\usepackage{tikz, tikz-cd}

\newcolumntype{C}{>{$}c<{$}}
\newcolumntype{L}{>{$}l<{$}}

\usepackage{todonotes}

\newcommand\OO{{\mathcal{O}}}

\newcommand{\Aut}{\operatorname{Aut}}

\newcommand{\GL}{\mathrm{GL}}

\newcommand{\KKK}{\operatorname{K3}}
\newcommand{\Kum}{\mathrm{Kum}}
\newcommand{\OG}{\mathrm{OG}}

\newcommand{\sm}{\textrm{sm}}
\newcommand{\topo}{\mathrm{top}}

\newcommand{\sL}{\mathcal{L}}
\newcommand{\sM}{\mathcal{M}}
\newcommand{\sO}{\mathcal{O}}
\newcommand{\sX}{\mathcal{X}}

\newcommand{\QQ}{\mathbb{Q}}
\newcommand{\CC}{\mathbb{C}}
\newcommand{\ZZ}{\mathbb{Z}}

\newtheorem{thm}{Theorem}[section]
\newtheorem{lem}[thm]{Lemma}

\newtheorem{prop}[thm]{Proposition}
\newtheorem{claim}[thm]{Claim}

\theoremstyle{definition}

\newtheorem{exmp}[thm]{Example}
\newtheorem{nota}[thm]{Notation}

\newtheorem{rmk}[thm]{Remark}

\theoremstyle{remark}

\begin{document}
\title[Numerically trivial automorphisms of hyperk\"ahler $4$-folds]{On numerically trivial automorphisms of compact hyperk\"ahler manifolds of dimension~4}
\date{\today}


\author{Chen Jiang}
\address{Shanghai Center for Mathematical Sciences \& School of Mathematical Sciences, Fudan University, Shanghai 200438, China}
\email{chenjiang@fudan.edu.cn}

\author{Wenfei Liu}
\address{Xiamen University, Xiamen, Fujian Province 361005, P. R. China}
\email{wliu@xmu.edu.cn}

\begin{abstract} 
We prove that the automorphism group of a compact hyperk\"ahler manifold of dimension $4$ acts faithfully on the cohomology ring.
 \end{abstract}
\maketitle

\numberwithin{equation}{section}

\section{Introduction}
Throughout the paper, we work over the complex number field $\mathbb{C}$.

A compact K\"{a}hler manifold $X$ is \emph{hyperk\"{a}hler} or \emph{irreducible symplectic} if it is simply connected and $H^0(X, \Omega^2_X)$ is spanned by an everywhere non-degenerate $2$-form. Compact hyperk\"{a}hler manifolds form an important class of manifolds with $c_1=0$, and their rich geometry attracts much attention from different areas of mathematics.

Compact hyperk\"{a}hler manifolds have even complex dimensions, with the two-dimensional ones being known as \emph{K3 surfaces}. It is well-known that $\Aut(X)$ acts faithfully on the cohomology ring $H^*(X, \QQ)$ for a K3 surface $X$ (see \cite[page~555]{PS71} or \cite[Proposition~VIII.11.3]{BHPV}). This fact enhances the Torelli theorem for K3 surfaces, and is useful in the construction of a universal family of marked K3 surfaces.

In general, for a compact complex space $X$, one considers the cohomological representation of the automorphism group $\varphi\colon \Aut(X) \rightarrow \GL(H^*(X, \QQ))$, and calls the kernel of $\varphi$ the group of \emph{numerically trivial automorphisms}, which is denoted by $\Aut_\QQ(X)$ in this paper. Thus $\Aut_\QQ(X)$ is trivial if $X$ is a K3 surface. It would be interesting to know if the same holds true for higher dimensional compact hyperk\"ahler manifolds. This is indeed the case for all the known examples: 
\begin{itemize}[leftmargin=*]
\item $\KKK^{[n]}$-type: the Hilbert scheme $X^{[n]}$ of $n$ points on a K3 surface $X$ (\cite{beauville2}), and their deformations (Lemma~\ref{lem: AutQ deform}).
\item $\Kum_n$-type: the generalized Kummer varieties $K_n(T)$ associated to a two-dimensional complex torus $T$, and their deformations (\cite{Ogu20}).
\item $\OG_6$-type and $\OG_{10}$-type: O'Grady's hyperk\"ahler $6$-folds and $10$-folds, and their deformations (\cite{MW17}). Note that, although the authors of \cite{MW17} did not explicitly state the triviality of $\Aut_\QQ(X)$ for O'Grady's $6$-folds, it follows readily from their description of the fixed locus of an automorphism acting trivially on $H^2(X, \QQ)$.
\end{itemize}
Compared to the K3 surface case, one difficulty in higher dimensions lies in the fact that there are several different topological types of compact hyperk\"ahler manifolds in each dimension $2n$ with $n\geq 2$, and the topological classification of compact hyperk\"ahler manifolds is still a very challenging open problem.

The main goal of this paper is to establish the triviality of $\Aut_\QQ(X)$ for any compact hyperk\"ahler manifold of dimension $4$, without referring to the explicit construction of these manifolds.
\begin{thm}\label{main}
Let $X$ be a compact hyperk\"ahler manifold of dimension $4$. Then $\Aut_{\mathbb Q}(X)$ is trivial. In other words, $\Aut(X)$ acts faithfully on $H^*(X, \QQ)$.
\end{thm}
We remark that Theorem~\ref{main} does not extend to singular symplectic varieties, as Example~\ref{ex} shows.

As suggested by the referee, we make some discussion on the automorphism group $\Aut_\mathbb{Z}(X)$ acting trivially on $H^*(X, \mathbb{Z})$, the cohomology ring with $\mathbb{Z}$-coefficients. Clearly, $\Aut_\mathbb{Z}(X)\subset\Aut_\mathbb{Q}(X)$, so the triviality of $\Aut_\mathbb{Q}(X)$ implies the triviality of $\Aut_\mathbb{Z}(X)$. In case the triviality of $\Aut_\mathbb{Q}(X)$ unfortunately fails, one might further ask for the triviality of $\Aut_\mathbb{Z}(X)$. One difficulty is that, for an element $g\in \Aut_\mathbb{Q}(X)$, one cannot conclude that $g\in \Aut_\mathbb{Z}(X)$ even if $g$ acts trivially on the torsion part of $H^*(X, \mathbb{Z})$. In principle, one needs to find explicitly the generators of $H^*(X, \mathbb{Z})$, in order to see whether each of them is fixed by $g$.

The paper is organized as follows: In Section~\ref{sec: pre} we recall the definitions of symplectic varieties and hyperk\"ahler manifolds, as well as several relevant subgroups of automorphisms. In Section \ref{sec: sym}, we study the fixed locus of a symplectic automorphism of a $4$-dimensional compact hyperk\"ahler manifold, and compute the Betti numbers of the quotient symplectic variety in terms of certain partial resolution. In Section~\ref{sec: O-trivial}, we construct line bundles with vanishing cohomology on a resolution of the quotient variety by an $\sO$-cohomologically trivial action, and this in turn gives a strong restriction on the topology when the resolution is a compact hyperk\"ahler manifold. In Section~\ref{sec: num trivial}, we prove the main theorem, and provide examples of singular symplectic varieties with nontrivial $\Aut_\QQ(X)$.

\medskip

{\bf \noindent Acknowledgment.} The authors would like to thank Shilin Yu for useful discussions and suggestions. 
Part of this work was done during the first author's visit to Xiamen University in May 2023, and the first author is grateful for the hospitality and support of Xiamen University.
We would like to thank the referee for useful comments and suggestions.

The first author was supported by National Key Research and Development Program of China \#2023YFA1010600, NSFC for Innovative Research Groups \#12121001, and National Key Research and Development Program of China \#2020YFA0713200.
The first author is a member of LMNS, Fudan University. The second author was supported by the NSFC (No.~11971399) and the Presidential Research Fund of Xiamen University (No.~20720210006).

\section{Preliminaries}\label{sec: pre}
In this section, we recall the definition of hyperk\"ahler manifolds as well as their singular version. We also define several subgroups of automorphisms considered in this paper. A {\it variety} means an irreducible and reduced separated complex space, or an integral separated scheme of finite type over $\mathbb{C}$ if one prefers the algebraic setting.

\subsection{Symplectic varieties}\label{sec: symp var}
Following \cite{Bea00} and \cite{BL22}, a \emph{symplectic variety} is a normal variety $X$ such that there is a non-degenerate closed holomorphic $2$-form $\sigma\in H^0(X_\sm, \Omega_{X_\sm}^2)$ on the smooth part $X_\sm$ and for any resolution of singularities $f\colon Y\rightarrow X$, the pullback $\pi^*\sigma$ extends to a holomorphic $2$-form on $Y$. A symplectic variety is called \emph{primitive} if it is compact K\"ahler with $H^1(X, \sO_X)=0$ and the holomorphic $2$-form is unique up to scalar. A compact \emph{hyperk\"ahler} or \emph{irreducible symplectic} manifold is a simply connected compact K\"ahler manifold that is symplectic and primitive.

\begin{exmp}\label{ex: symp}
 \begin{enumerate}[leftmargin=*]
 \item As seen in the introduction, all the known examples of compact hyperk\"ahler manifolds are divided in to several types according to their constructions: $\KKK^{[n]}$-type, $\Kum_n$-type, $\OG_6$-type, and $\OG_{10}$-type.
 \item (\cite[Example~3.2]{BL22}) Let $X$ be a primitive symplectic variety.
 \begin{enumerate}
  \item For a bimeromorphic contraction $f\colon X\rightarrow Y$ onto a normal variety, the target $Y$ is also a primitive symplectic variety.
  \item Any quotient of $X$ by a finite group of symplectic automorphisms preserving the symplectic structure of $X$ is again a primitive symplectic variety (see below for the definition of symplectic automorphisms).
 \end{enumerate}
 \end{enumerate}
\end{exmp}

\subsection{Automorphisms acting trivially on certain cohomology groups}
Let $X$ be a compact normal variety. We define 
\begin{itemize}
 \item the \emph{full automorphism group}
 \[
 \Aut(X)=\{g \mid g\colon X\rightarrow X \text{ is a biholomorphic map}\};
 \]
 \item the \emph{group of numerically trivial automorphisms}
 \[
 \Aut_\QQ(X)=\{g\in \Aut(X) \mid g^*|_{ H^*(X, \QQ)}=\text{id}\};
 \]
 \item the \emph{group of $\sO$-cohomologically trivial automorphisms}
 \[
 \Aut_\sO(X)=\{g\in \Aut(X) \mid g^*|_{H^*(X, \sO_X)}=\text{id}\};
 \]
 \item for a fixed $0\neq \sigma\in H^0(X_\sm,\Omega_{X_\sm}^2)$, the \emph{group of symplectic automorphisms (with respect to $\sigma$)} 
 \[
 \Aut(X, \sigma)=\{g\in \Aut(X) \mid g^*\sigma=\sigma\}.
 \]
\end{itemize}
If $X$ is a compact K\"ahler manifold, then $H^*(X, \sO_X)$ is direct summand of $H^*(X, \CC)=H^*(X, \QQ)\otimes \CC$ by the Hodge theory, and $\bar \sigma \in H^2(X, \sO_X)$. So we have natural inclusions
$$\Aut_\QQ(X)\subset\Aut_\sO(X)\subset \Aut(X,\sigma)\subset \Aut(X).
$$

\section{Quotients of compact hyperk\"ahler manifolds by symplectic automrophisms}\label{sec: sym}
We state several well-known facts about symplectic automrophisms of a compact hyperk\"ahler manifold and the quotient thereof, and at the same time fix the notation that will be used throughout the paper.
\begin{nota}\label{nota: symp}
Let $X$ be a compact hyperk\"ahler manifold and fix a non-zero $\sigma\in H^0(X, \Omega_{X}^2)$. Let $g\in \Aut(X, \sigma)$ be a symplectic automorphism of prime order $p$. We recall some easy facts about the fixed locus $X^g$ and the quotient variety $Y:=X/\langle g\rangle$.
\begin{enumerate}
 \item The connected components of the fixed locus $X^g$ are holomorphic symplectic manifolds (\cite[Proposition~5.2]{Ogu20}). In particular, if $\dim X=4$, then the fixed locus $X^g$ is the disjoint union of isolated points, K3 surfaces, and $2$-dimensional complex tori, the numbers of which are denoted by $m$, $k$, and $t$ respectively. 
 \item The quotient variety $Y:=X/\langle g\rangle$ has symplectic singularities; see Example~\ref{ex: symp} (2b) or \cite[(2.3)]{Bea00}. 
 In codimension $2$, $Y$ has at most ADE singularities by the classification of $2$-dimensional symplectic singularities. 
 Let $\theta: W\to Y$ be the crepant partial resolution of all codimension $2$ singularities. Then $W$ is a \emph{primitively symplectic orbifold} in the sense of \cite{FM21}. By \cite[Proposition 3.4]{Fu06}, $Y$ admits a symplectic resolution if and only if $W$ is smooth. 
\end{enumerate}
\end{nota}




\begin{lem}\label{lem: bYbW}
Keep the Notation~\ref{nota: symp} and assume that $\dim X=4$. Then the following relations on Betti numbers hold:
 \begin{align*}
 b_2(W){}&=b_2(Y)+(p-1)(k+t);\\
 b_3(W){}&=b_3(Y)+4(p-1)t;\\
 b_4(W){}&=b_4(Y)+(p-1)(22k+6t).
\end{align*}
\end{lem}

\begin{proof}
Denote by $S_i$ $(1\leq i\leq k+t)$ the codimension $2$ singular loci on $Y$. Recall that any point in $S_i$ is a cyclic quotient singularity of type $\frac{1}{p}(0,0,1,-1)$. Denote by $\theta^{-1}(S_i)= E_i$, then $\theta|_{E_i}: E_i\to S_i$ is a locally trivial fibration whose fiber is
a chain of $p-1$ smooth rational curves, and $E_i$ has $p-1$ irreducible components. By Lemmas~\ref{lem:comparing bi of bir map1} and \ref{lem:comparing bi of bir map2}, we have
\begin{equation}\label{eq: betti W}
\begin{split}
b_j(W){}&=b_j(Y)+ b_{j}(\bigsqcup_i E_i)-b_j(\bigsqcup_i S_i)\\
{}& =b_j(Y)+(p-1)\sum_{i=1}^{l} b_{j-2}(S_i),
\end{split}
\end{equation}
where for the last equality we use the Leray--Hirsch theorem (\cite[Theorem~4D.1]{Hat02}) to $E_i\to S_i$.
Finally, recall that
\begin{equation}\label{eq: betti S}
 (b_0(S_i), b_1(S_i), b_2(S_i))=(1, 0, 22) \text{ or } (1, 4, 6)
\end{equation} 
depending on whether $S_i$ is a K3 surface or a $2$-dimensional complex torus. The proof is complete by plugging \eqref{eq: betti S} into \eqref{eq: betti W}.
\end{proof}

\begin{lem}\label{lem:comparing bi of bir map1}
Let $f: \tilde Z \to Z$ be a holomorphic map between compact complex analytic spaces. Let $D\subset Z$ be a closed analytic subspace and $E= f^{-1}(D)$.
Suppose that for each integer $k$, the pullback map
$f^*: H^k(Z, \mathbb{Q}) \to H^k(\tilde Z, \mathbb{Q})$ is injective. Then for each integer $k$, there is a short exact sequence on cohomology groups:
$$
0\to H^k(Z, \mathbb{Q}) \to H^k(\tilde Z, \mathbb{Q}) \oplus H^k(D, \mathbb{Q}) \to H^k(E, \mathbb{Q})\to 0.
$$
\end{lem}
\begin{proof}
 By \cite[Proposition~3.1]{FPR15} (cf. \cite[Corollary-Definition~5.37]{PS08}), there is a natural Mayer--Vietoris long exact sequence 
 $$
\dots \to H^{k-1}(E)\to H^k(Z) \to H^k(\tilde Z) \oplus H^k(D) \to H^k(E)\to H^{k+1}(Z) \to \cdots.
$$
Here the map $H^k(Z) \to H^k(\tilde Z)$ is given by the pullback $f^*$.
Therefore, the above long exact sequence breaks into short exact sequences.
\end{proof}

\begin{lem}\label{lem:comparing bi of bir map2}
Let $f: \tilde Z \to Z$ be a proper
surjective bimeromorphic morphism between almost K\"ahler V-manifolds (see \cite[\S2.5]{PS08}). Then for each integer $k$, the pullback map
$f^*: H^k(Z, \mathbb{Q}) \to H^k(\tilde Z, \mathbb{Q})$ is injective.
\end{lem}
\begin{proof}
Take a bimeromorphic morphism $g: Y\to \tilde Z$ such that $Y$ is bimeromorphic to a K\"ahler manifold, then by \cite[Proof of Theorem~2.43]{PS08}, 
$(f\circ g)^*: H^k(Z) \to H^k(Y)$ is injective, which implies that $f^*$ is injective.
\end{proof}

 \section{Existence of line bundles with vanishing cohomology}\label{sec: O-trivial}
 We observe in Lemma~\ref{lem: L=0} the existence of line bundles with vanishing cohomology on a resolution of the quotient variety by an $\sO$-cohomologically trivial action. This can be used as in Proposition~\ref{prop: inv2} to deduce strong topological restrictions when the resolution is a compact hyperk\"ahler manifold.
\begin{lem}\label{lem: L=0}
Let $X$ be a compact complex manifold, and $g\in \Aut_\sO(X)$ an $\sO$-cohomologically trivial automorphism of prime order $p$. Let $\theta\colon W\rightarrow X/\langle g\rangle$ be a resolution of the quotient variety. Then there is a line bundle $\sL$ on $W$ such that $H^i(W, \sL)=0$ for any $i\geq 0$.
\end{lem}
\begin{proof}

Denote $Y=X/\langle g\rangle$, and let $\pi\colon X\rightarrow Y$ be the quotient map. Then $\pi_*\sO_X = \bigoplus_{j=0}^{p-1} \sM_j$, where for each $0\leq j\leq p-1$, $\sM_j$ is the rank $1$ eigensheaf corresponding to the eigenvalue $e^{2\pi j\sqrt{-1}/p}$. In particular, $\sM_0=\sO_Y$. Since $g$ acts trivially on $H^i(X, \sO_X)$, we have $H^i(X, \sO_X) = H^i(Y, \sO_Y)$ for any $i\geq 0$. It follows that $H^i(Y, \sM_j) = 0$ for $i\geq 0$ and $0<j<p$. 

\begin{claim}\label{clm: L}
The double dual $\sL:=(\theta^*\sM_1)^{\vee\vee}$ is invertible and $\theta_*\sL= \sM_1$. 
\end{claim}
\begin{proof}[Proof of Claim~\ref{clm: L}]
Since $\sL$ is reflexive of rank 1 on a smooth variety, it is invertible by \cite[Proposition~1.9]{Har80} or \cite[II.1.1.15]{OSS80}. The natural homomorphism $\theta^*\sM_1\rightarrow (\theta^*\sM_1)^{\vee\vee} = \sL$ induces an injective homomorphism $\sM_1\rightarrow \theta_*\sL$. As a direct summand of the Cohen--Macaulay module $\pi_*\sO_X$, the sheaf $\sM_1$ is Cohen--Macaulay. In particular, $\sM_1$ satisfies Serre's $S_2$ condition. It follows that the injective homomorphism $\sM_1\rightarrow \theta_*\sL$ is also surjective, hence an isomorphism.
\end{proof}

Finally, we have 
\[
H^i(W, \sL) = H^i(Y, \theta_*\sL) = H^i(Y, \sM_1)=0 
\]
for any $i\geq 0$, where the first equality follows from the fact that $Y$ has at most quotient singularities. The proof of the lemma is now complete.
\end{proof}


The following result is similar to \cite[Proposition~5.6]{BS22}, and might be of independent interest in the topological classification of $4$-dimensional compact hyperk\"ahler manifolds. 
\begin{prop}\label{prop: inv2}
Let $W$ be a compact hyperk\"ahler manifold of dimension $4$.
 If there exists a line bundle $\sL$ on $W$ such that $\chi(W, \sL)=0$, then 
 the (essential) Chern and Betti numbers of $W$ are among one of the following: 
 {
\begin{longtable}{LLLLL}
\caption{Topological data of $W$}\label{tab1}\\
\hline
\text{\rm No.} & c_2^2(W) & c_4(W) & b_2(W) & b_3(W) \\
\hline
\endfirsthead
\\
\hline 
\text{No.} & c_2^2(W) & c_4(W) & b_2(W) & b_3(W) \\
\endhead
\hline
\endfoot

\hline \hline
\endlastfoot

1& 828& 324 & 23 & 0\\
2& 756& 108 & 7 & 8\\
3& 756& 108 & 6 & 4\\
4& 756& 108 & 5 & 0 \\
\hline
\end{longtable}
}
Here $c_2^2(W)=\int c_2^2(W)$ and $c_4(W)=\int c_4(W)$ by abusing the notation.
\end{prop}

\begin{proof}
By \cite[Theorem~5.2 and (5.12)]{Nieper}, we can express the Riemann--Roch formula as the following:
\begin{align}
 \chi(W, \sL)=\chi(W, \OO_W) + \frac{1}{720}\left(\frac{7}{2}c_2^2-2c_4\right)\lambda(\sL)+\frac{1}{720}
\left(\frac{7}{8}c_2^2-\frac{1}{2}c_4\right)\lambda(\sL)^2,\label{eq:RR1}
\end{align} 
 where $$
\lambda(\sL):=\begin{cases}\frac{48\int \exp(\sL)}{ \int c_{2}(W) \exp(\sL)} & \text{if well-defined;}\\ 0 & \text{otherwise}\end{cases}
$$ is the {\it characteristic value} of $\sL$ (see \cite[Definition~17]{Nieper}).
By the fact that 
$$
3c_2^2-c_4=720\chi(W, \OO_W)=2160,
$$
we can simplify \eqref{eq:RR1} as
\begin{align*}
 \chi(W, \sL)=3 + \left(\frac{7}{2}-\frac{1}{864}c_4\right)\lambda(\sL)+ \left(\frac{7}{8} -\frac{1}{3456}c_4\right)\lambda(\sL)^2. 
\end{align*} 
It is clear that $\lambda(\sL)\in \mathbb{Q}$ by definition. By the assumption that $\chi(W, \sL)=0$, we infer that 
$$
\Delta: =\left(\frac{7}{2}-\frac{1}{864}c_4\right)^2-12\left(\frac{7}{8} -\frac{1}{3456}c_4\right)
$$
is a square of a rational number, which gives a strong restriction on the value of $c_4$. 

Recall that Guan \cite[Corollary~1]{guan01} gave a list of all possible Betti numbers of $W$
 (see also \cite[Theorem 3.6 and Corollary 3.7]{BD22}), namely, there are 53 possible values for the pair $(b_2(W), b_3(W))$. 
On the other hand, we have the following relations between Chern numbers and Betti numbers (see for example \cite[Proof of Theorem~2]{guan01}):
\begin{align*}
 c_4(W) {}&= 48+12b_2(W)-3b_3(W),\\
 c_2^2(W) {}& = 736+4b_2(W)-b_3(W).
\end{align*} 
Now it is straightforward to check that the only possible values making $\Delta$ a square of a rational number are those listed in Table~\ref{tab1}.
\end{proof}
\begin{rmk}
The topological data No.~1 and No.~2 in Table~\ref{tab1} are realized by $\KKK^{[2]}$-type and $\Kum_2$-type hyperk\"ahler manifolds respectively (\cite{BD22}). However, it is difficult to detect the actual existence of a line bundle with trivial Euler characteristic on a given compact hyperk\"ahler manifold.
\end{rmk}


\section{Triviality of numerically trivial automorphism group}\label{sec: num trivial}




 
The following two lemmas about $\Aut_\QQ(X)$ for compact hyperk\"ahler manifolds are well-known to experts. 

The first one is about the finiteness of $\Aut_\QQ(X)$.
\begin{lem}\label{lem: finite AutQ}
 Let $X$ be a compact K\"ahler manifold with $H^0(X, T_X)=0$. Then $\Aut_\QQ(X)$ is finite. In particular, $\Aut_\QQ(X)$ is finite if $X$ is hyperk\"ahler.
\end{lem}
\begin{proof}
Since $H^0(X, T_X)=0$, $\Aut_{\mathbb Q}(X)$ is discrete. Since $\Aut_{\mathbb Q}(X)$ fixes the K\"ahler classes by definition, it is finite (\cite{Fuj78, Lie78}). 
\end{proof}

The second lemma is about the deformation invariance of $\Aut_\QQ(X)$. 
\begin{lem}\label{lem: AutQ deform}
Let $f\colon \sX\rightarrow B$ be a family of compact hyperk\"ahler manifolds over a connected complex manifold $B$, and denote by $X_t=f^{-1}(t)$ the fiber over a point $t\in B$. Then, for $0\in B$ and $g\in\Aut_\QQ(X_{0})$, there is an automorphism $\tilde g\in \Aut(\sX)$ such that $\tilde g$ preserves each fiber $X_t$ of $f$, and $\tilde g|_{X_t}$ lies in $\Aut_\QQ(X_t)$.
\end{lem}
\begin{proof}
Since $g$ is numerically trivial and $H^2(X,\ZZ)$ has no torsion, $g$ induces the trivial action on $H^2(X, \ZZ)$. By \cite[Theorem~2.1]{HT13}, there is an automorphism $\tilde g\in \Aut(\sX)$ such that it preserves each $X_t$ and $\tilde g|_{X_0}=g$. Since the induced action of $\tilde g|_{X_t}$ on the cohomology ring $H^*(X_t,\QQ)$ is locally constant with respect to $t\in B$ and $B$ is connected, the triviality of $g^*|_{H^*(X_0,\QQ)}$ implies that of $(\tilde g|_{X_t})^*|_{H^*(X_t,\QQ)}$ for any $t\in B$.
\end{proof}

 
From now on, let $X$ be a $4$-dimensional compact hyper\"ahler manifold, and $g\in\Aut_\QQ(X)$ a numerically trivial automorphism of prime order $p$, which is automatically symplectic with respect to any $\sigma\in H^0(X, \Omega_X^2)$. Consider the following diagram
\[
\begin{tikzcd}
X \arrow[rd, "\pi"] & & W \arrow[ld, "\theta"']\\
& Y=X/\langle g \rangle &
\end{tikzcd}
\]
where $\pi\colon X\rightarrow Y=X/\langle g\rangle$ is the quotient map, and $\theta\colon W\rightarrow Y$ the crepant partial resolution that resolves the codimension $2$ singularities of $Y$. We shall compare the topological data of $X, Y, W$, in order to obtain a contradiction. Hereby, it is crucial to describe the fixed locus $X^\sigma$. As in Notation~\ref{nota: symp}, we denote the numbers of points, K3 surfaces, and $2$-dimensional complex tori in $X^g$ by $m, k$, and $t$ respectively. 
\begin{lem}\label{lem: km=0} 
The following assertions hold:
\begin{enumerate}
\item $m=k=0$. Namely, there are no isolated points or K3 surfaces in the fixed locus $X^g$.
\item $W$ is smooth, and hence a $4$-dimensional compact hyperk\"ahler manifold.
\item $\chi_\topo(W) = \chi_\topo(X)=0$, where $\chi_\topo$ denotes the topological Euler characteristic.
\end{enumerate}

\end{lem}

\begin{proof}
(1) By the orbifold Salamon relation \cite[Theorem~3.6]{FM21}, we have $$b_4(W)+b_3(W)-10b_2(W)=46+s,$$
where $s$ is the contribution of singularities.
Moreover, by our construction, $W$ has exactly $m$ cyclic quotient singularities of index $p$, so $s=-m(p-1)$ by the computation in \cite[Example~3.8]{FM21}.
On the other hand, we have $$b_4(X)+b_3(X)-10b_2(X)=46$$ by the usual Salamon relation \cite{Salamon}
and $b_j(Y)=b_j(X)$ for all $j$ by the numerical triviality of $g$.
 Hence by Lemma~\ref{lem: bYbW}, we have
 $$-m(p-1)=(p-1)(22k+6t+4t-10k-10t),$$ which implies that $(m+12k)(p-1)=0.$ Since $p-1>0$, we infer that $m=k=0$.

(2) The smoothness of $W$ follows from the construction of $W$ and $m=0$.

(3) By (1), the fixed locus $X^g$ is the disjoint union of $2$-dimensional complex tori $S_i$ ($1 \leq i\leq t$). Since $g\in \Aut_{\mathbb{Q}}(X)$, by the topological Lefschetz fixed point formula, we have
\[
\chi_\topo(X) = \chi_\topo(X^g) = \sum_{ i=1 }^t \chi_\topo(S_i) = 0. 
\]
On the other hand, by Lemma~\ref{lem: bYbW}, we have 
\begin{align*}
 b_2(W){}&=b_2(Y)+(p-1)t,\\
 b_3(W){}&=b_3(Y)+4(p-1)t,\\
 b_4(W){}&=b_4(Y)+6(p-1)t.
\end{align*}
Taking the alternating sum of the Betti numbers, we obtain 
\[
\chi_\topo(W) = \chi_\topo(Y) =\chi_\topo(X)=0,
\]
where the second equality is again by the numerical triviality of $g$.
\end{proof}

\begin{proof}[Proof of Theorem~\ref{main}]
Suppose on the contrary that $\Aut_{\mathbb Q}(X)$ is nontrivial, then there is an automorphism $g\in \Aut_{\mathbb Q}(X)$ of prime order. The quotient variety $Y=X/\langle g\rangle$ admits a symplectic resolution $W$ with $c_4(W)=\chi_\topo(W)=0$ by Lemma~\ref{lem: km=0}. 
On the other hand, $W$ admits a line bundle $\sL$ with $\chi(W, \sL)=0$ by Lemma~\ref{lem: L=0}, which contradicts Proposition~\ref{prop: inv2}.
\end{proof}

Finally, in contrast to Theorem~\ref{main}, we provide examples of \emph{singular} primitive symplectic varieties whose $\Aut_\QQ(X)$ is nontrivial. The automorphisms involved have been treated from other perspectives in \cite{beauville2, Ogu20}.
\begin{exmp}\label{ex}
\begin{enumerate}[leftmargin=*] 
\item Let $T$ be a $2$-dimensional complex torus, and $X_2=T/\langle -1 \rangle$ the quotient of $T$ by the involution $-1$. Then $X_2$ is a (singular) Kummer surface with $16$ $A_1$-singularities. Let $\tau_\epsilon$ be a translation of $T$ by a $2$-torsion element $\epsilon\in T[2]$. Then $\tau_\epsilon$ commutes with $-1$, and hence descends to an involution of $X_2$, denoted by $\bar \tau_\epsilon$. Note that, translations act trivially on $H^*(T, \QQ)$, and it follows that the induced automorphism $\bar \tau_\epsilon$ acts trivially on $H^*(X_2, \QQ)$, which can be viewed as a subspace of $H^*(T, \QQ)$ via the pullback by the quotient map. Therefore, $G:=\{\bar\tau_\epsilon \mid \epsilon\in T[2]\}\cong(\ZZ/2\ZZ)^4$ is a subgroup of $\Aut_\QQ(X_2)$.

\item The above construction of Kummer surfaces can be generalized to higher dimensions as follows. For a given integer $n\geq 2$, consider the following sub-torus of $T^{n+1}$:
\[
T(n):=\left\{(P_1, P_2, \dots, P_{n+1})\in T^{n+1} \,\bigg|\, \sum_{i=1}^{n+1}P_i =0\right\}.
\]
We have $T(n)\cong T^n$, and it is stable under the action by the permutation group $\mathfrak{S}_{n+1}$ on the $n+1$ coordinates of $T^{n+1}$, and by $G\cong T[n+1]\cong(\ZZ/(n+1)\ZZ)^4$, where $G$ consists of simultaneous translations by an $(n+1)$-torsion element $\epsilon\in T[n+1]$ on each coordinate.

Now set $X_{2n} = T(n)/\mathfrak{S}_{n+1}$. Then $X_{2n}$ is a singular irreducible symplectic variety, since the generalized Kummer variety $K_{n}(T)$ is a symplectic resolution of $X_{2n}$ (\cite[\S2]{Ogu20}). Since the actions of $G$ and $\mathfrak{S}_{n+1}$ on $T(n)$ commute, the action of $G$ descends to $X_{2n}$. Now, as in the surface case above, the triviality of the $G$-action on $H^*(T(n), \QQ)$ implies the triviality of the $G$-action on $H^*(X_{2n}, \QQ)$, and hence $G$ is a subgroup of $\Aut_\QQ(X_{2n})$.

\end{enumerate}

\end{exmp}





\begin{thebibliography}{BHPV04}

\bibitem[BHPV04]{BHPV} W. Barth, K. Hulek, C. Peters, A. Van de Ven, {\it Compact complex surfaces}, second enlarged edtion, Ergebnisse der Mathematik und ihrer Grenzgebiete, Vol 4, Springer Verlag, 2004.

\bibitem[BL22]{BL22} B. Bakker, C. Lehn, {\it The global moduli theory of symplectic varieties}, J. Reine Angew. Math. 790 (2022), 223--265


\bibitem[Bea83]{beauville2}A.~Beauville, {\it Some remarks on K\"ahler manifolds with $c_1=0$}, in: Classification of algebraic and analytic manifolds (Katata, 1982), pp. 1--26, Progr. Math., 39, Birkh\"auser Boston, Boston, MA, 1983.

\bibitem[Bea00]{Bea00} A. Beauville, {\it Symplectic singularities}, Invent. Math. 139 (2000), no.3, 541--549.


\bibitem[BS22]{BS22} T.~Beckmann, J.~Song, {\it Second Chern class and Fujiki constants of hyperk\"ahler manifolds}, arXiv:2201.07767. 


\bibitem[BD22]{BD22} P.~Beri, O.~Debarre, {\it On the Hodge and Betti numbers of hyper-K\"ahler manifolds}, Milan J. Math. 90 (2022), no. 2, 417--431. 


\bibitem[FPR15]{FPR15} M.~Franciosi, R.~Pardini, S.~Rollenske, {\it Computing invariants of semi-log-canonical surfaces}, Math. Z. 280 (2015), no. 3-4, 1107--1123.

\bibitem[Fu06]{Fu06} B.~Fu, {\it A survey on symplectic singularities and symplectic resolutions}, Ann. Math. Blaise Pascal 13 (2006), no.2, 209--236.


 \bibitem[FM21]{FM21} L.~Fu, G.~Menet, {\it On the Betti numbers of compact holomorphic symplectic orbifolds of dimension four}, Math. Z. 299 (2021), no. 1-2, 203--231.

 \bibitem[Fuj78]{Fuj78}
A.~Fujiki, {\it On automorphism groups of compact K\"ahler manifolds}, Invent. Math. 44 (1978), no. 3, 225--258. 

 \bibitem[Gua01]{guan01} D. Guan, {\it On the Betti numbers of irreducible compact hyperk\"ahler manifolds of complex dimension four}, Math. Res. Lett. 8 (2001), no. 5--6, 663--669.

\bibitem[Har80]{Har80}
R.~Hartshorne, {\it Stable reflexive sheaves}, Math. Ann. 254 (1980), no. 2, 121--176.



\bibitem[HT13]{HT13}
B. Hassett, Y. Tschinkel, {\it Hodge theory and Lagrangian planes on generalized Kummer fourfolds}, Moscow Math. J. 13 (2013), no. 1, 33--56.


\bibitem[Hat02]{Hat02} A.~Hatcher, {\it Algebraic topology},
Cambridge University Press, Cambridge, 2002. xii+544 pp.


\bibitem[Lie78]{Lie78}
 D. Lieberman, {\it Compactness of the Chow scheme: applications to automorphisms and deformations of K\"ahler manifolds}, Fonctions de plusieurs variables complexes, III (S\'em. Francois Norguet, 1975--1977), pp. 140--186, Lecture Notes in Math., 670, Springer, Berlin, 1978. 

\bibitem[MW17]{MW17}
 G.~Mongardi, M.~Wandel, {\it Automorphisms of O'Grady's manifolds acting trivially on cohomology}, Algebr. Geom. 4 (2017), no. 1, 104--119.

 \bibitem[Nie03]{Nieper} M.~A.~Nieper, 
{\it Hirzebruch--Riemann--Roch formulae on irreducible symplectic K\"ahler manifolds}, 
J. Algebraic Geom. 12 (2003), no. 4, 715--739.





 \bibitem[Ogu20]{Ogu20} K.~Oguiso, {\it No cohomologically trivial nontrivial automorphism of generalized Kummer manifolds}, Nagoya Math. J. 239 (2020), 110--122.

\bibitem[OSS80]{OSS80} C.~Okonek, M.~Schneider, H.~Spindler, {\it Vector bundles on complex projective spaces}, Progress in Mathematics, 3, Birkh\"auser, Boston, Mass., 1980. 


\bibitem[PS08]{PS08} C.~A.~M.~Peters, J. ~H.~M.~Steenbrink, {\it Mixed Hodge structures},
Ergeb. Math. Grenzgeb. (3), Vol 52, Springer-Verlag, Berlin, 2008. 


\bibitem[PS71]{PS71}
I. Piate\u{c}kii-Shapiro, I. \u{S}afarevi\u{c}, {\it A Torelli theorem for algebraic surfaces of type {${\rm K}3$}}, Math. USSR, Izv. 5 (1971), 547--588. 


\bibitem[Sal96]{Salamon} S.~Salamon, {\it On the cohomology of K\"ahler and hyperk\"ahler manifolds}, Topology 35 (1996), no. 1, 137--155.



\end{thebibliography}
\end{document}